\documentclass[american]{article}
\usepackage{mathptmx}
\usepackage[T1]{fontenc}
\usepackage[latin9]{inputenc}
\usepackage{units}
\usepackage{amsmath}
\usepackage{amsthm}
\usepackage{amssymb}
\usepackage[all]{xy}

\makeatletter
\numberwithin{equation}{section}
\numberwithin{figure}{section}
\theoremstyle{plain}
\newtheorem{thm}{\protect\theoremname}[section]
  \theoremstyle{definition}
  \newtheorem{defn}[thm]{\protect\definitionname}
  \theoremstyle{remark}
  \newtheorem{rem}[thm]{\protect\remarkname}
  \theoremstyle{plain}
  \newtheorem{cor}[thm]{\protect\corollaryname}
  \theoremstyle{plain}
  \newtheorem{question}[thm]{\protect\questionname}
  \theoremstyle{plain}
  \newtheorem{prop}[thm]{\protect\propositionname}

\@ifundefined{date}{}{\date{}}
\makeatother

\usepackage{babel}
  \providecommand{\corollaryname}{Corollary}
  \providecommand{\definitionname}{Definition}
  \providecommand{\propositionname}{Proposition}
  \providecommand{\questionname}{Question}
  \providecommand{\remarkname}{Remark}
\providecommand{\theoremname}{Theorem}

\begin{document}

\title{On Semi-Rational Groups}

\author{Tzoor Plotnikov}

\maketitle
\begin{sloppy}
{\let\thefootnote\relax\footnotetext{The author acknowledges the support of ISF 1031/17 grant of Ori Parzanchevski}}
\begin{abstract}
A finite group is called \emph{semi-rational }if the distribution
induced on it by any word map is a virtual character. In \cite{amit2011characters}
Amit and Vishne give a sufficient condition for a group to be semi-rational,
and and ask whether it is also necessary. We answer this in the negative,
by exhibiting two new criteria for semi-rationality, each giving rise
to an infinite family of semi-rational groups which do not satisfy
the Amit-Vishne condition. On the other hand, we use recent work of
Lubotzky to show that for finite simple groups the Amit-Vishne condition
is indeed necessary, and we use this to construct the first known
example of an infinite family of non-semi-rational groups.
\end{abstract}

\section{Introduction\label{sec:Introduction}}

For a finite group $G$ and a word $w=x_{i_{1}}^{n_{1}}x_{i_{2}}^{n_{2}}\cdots x_{i_{\ell}}^{n_{\ell}}$
in the free group on $r$ generators one can define the ``word map''
$w:G^{r}\rightarrow G$ by
\[
w(g_{1},\cdots,g_{r})=g_{i_{1}}^{n_{1}}g_{i_{2}}^{n_{2}}\cdots g_{i_{\ell}}^{n_{\ell}},
\]
namely, substituting each $x_{i}$ by $g_{i}$. One defines $N_{w,G}:G\rightarrow\mathbb{C}$
by 
\[
N_{w,G}(g)=|\{(g_{1},\cdots,g_{r})\in G^{r}:\ w(g_{1},\cdots,g_{r})=g\}|,
\]
the distribution which $w$ induces on $G$.

For every automorphism $\alpha\in Aut(G)$ and every $g\in G$ there
is a bijection between the solution set of $w=g$ and the solution
set of $w=\alpha(g)$ given by $(g_{1},\cdots,g_{r})\mapsto(\alpha(g_{1}),\cdots\alpha(g_{r}))$.
In particular, $N_{w,G}$ is a class function of $G$, and can be
written as 
\[
N_{w,G}=\sum_{\chi\in Irr(G)}N_{w,G}^{\chi}\cdot\chi
\]
where by orthogonality of characters
\[
N_{w,G}^{\chi}=\langle N_{w,G},\chi\rangle=\frac{1}{|G|}\sum_{g\in G}N_{w,G}(g)\overline{\chi}(g)=\frac{1}{|G|}\sum_{(g_{1},\cdots,g_{r})\in G^{r}}\overline{\chi}(w(g_{1},\cdots,g_{r})),
\]
with $\overline{\chi}$ being the complex conjugate of $\chi$. If
for every irreducible character $N_{w,G}^{\chi}\in\mathbb{N}$ then
$N_{w,G}$ is a character of the group $G$, and if $N_{w,G}^{\chi}\in\mathbb{Z}$
then $N_{w,G}$ is a difference of characters and we call it a generalized
character, or a virtual character of $G$.
\begin{defn}
\label{def:SR_group}A finite group $G$ is called semi-rational if
$N_{w,G}$ is a generalized character for every $r\in\mathbb{N}$
and every $w\in F_{r}$.
\end{defn}
In \cite{amit2011characters} Amit and Vishne provided several results
on word maps from character theoretic point of view. They used an
argument by Stanley (\cite{stanley1986enumerative}, Exercise 7.69.j)
to show
\begin{thm}[\cite{amit2011characters}, Proposition 3.2]
\label{thm:AV-fourier-coeff} For every $w\in F_{r}$ and $\chi\in Irr(G)$
one has $N_{w,G}^{\chi}\in\mathbb{Z}[\{\chi(g)|\ g\in G\}]$. In particular,
$N_{w,G}^{\chi}\in\mathbb{Z}[\omega]$ where $\omega$ is a primitive
root of unity of order $|G|$.\end{thm}
\begin{rem}
\label{rem:AV-fourier-coeff}The theorem Amit and Vishne prove in
their work actually states that $N_{w,G}^{\chi}\in\mathbb{Z}[\{\psi(g)|\ g\in G,\ \psi\in Irr(G)\}]$,
but examination of the proof shows that the stronger version given
here is still viable, and we will use this stronger version in Proposition
\ref{prop:character-condition-SR}.
\end{rem}
They use this theorem to prove a necessary and sufficient condition
for a group to be semi-rational:
\begin{thm}[\cite{amit2011characters}, Corollary 3.3]
\label{thm:AV} $N_{w,G}$ is a generalized character if and only
if for every $g,h\in G$ generating the same cyclic subgroup $N_{w,G}(g)=N_{w,G}(h)$.
\end{thm}
They get as a corollary
\begin{cor}[\cite{amit2011characters}, Proposition 3.5]
\label{cor:AV-condition} If $g,h\in G$ lie in the same orbit of
$Aut(G)$ whenever they generate the same cyclic subgroup then $G$
is semi-rational.
\end{cor}
We call the condition of Corollary \ref{cor:AV-condition} the \emph{Amit-Vishne
condition}. In their article Amit and Vishne asked 
\begin{question}[\cite{amit2011characters}, Question 3.9]
\label{ques:AV-condition} Is the Amit-Vishne condition also necessary
for a group to be semi-rational?
\end{question}
We provide in Section \ref{sub:The-Counter-Examples} two infinite
families of semi-rational groups which do not satisfy the Amit-Vishne
condition, answering Question \ref{ques:AV-condition} in the negative.
These examples are based on Proposition \ref{prop:character-condition-SR}
and Corollary \ref{cor:Ori-condition}, which gives new sufficient
conditions for semi-rationality. We also use those new conditions
to show that every group of order $pq$ is semi-rational.

On the other hand, we prove in Section \ref{sub:The-Finite-Simple}
that the Amit-Vishne condition is indeed necessary for finite simple
groups, and use this result to provide the first example of an infinite
family of groups which are not semi-rational.

In Section \ref{sec:Further-Questions} we provide several further
questions and conjectures on the subject.

\section{On The Amit-Vishne Condition\label{sec:On-The-Amit-Vishne}}

\subsection{The Counter Examples\label{sub:The-Counter-Examples}}

In \cite{amit2011characters}, Amit and Vishne ask whether the condition
in Corollary \ref{cor:AV-condition} is also necessary for a group
to be semi-rational. We present in this section two infinite families
of semi-rational groups which do not satisfy the Amit-Vishne condition.
In order to prove the semi-rationality of the groups, we prove new
conditions for semi-rationality.

The first one is:
\begin{prop}
\label{prop:character-condition-SR} For $G$ a finite group, if every
irreducible character of degree $\geq2$ takes values only in $\mathbb{Z}$
then $G$ is semi-rational.\end{prop}
\begin{proof}
Using Theorem \ref{thm:AV-fourier-coeff}, it is enough to show that
for every word $w\in F_{r}$ and every one-dimensional character $\chi$
of $G$, one has $N_{w,G}^{\chi}\in\mathbb{Z}$.

For $\chi$ and $w$ as above, denote by $K$ the kernel of $\chi$,
and write $\tilde{N}_{w,G}(gK)=\frac{1}{|K|}\sum_{k\in K}N_{w,G}(gk)$.
Then one has
\[
N_{w,G}^{\chi}=\left\langle N_{w,G},\chi\right\rangle =\frac{1}{|G|}\sum_{g\in G}N_{w,G}(g)\overline{\chi}(g)=\frac{1}{|\nicefrac{G}{K}|}\sum_{gN\in\nicefrac{G}{K}}\tilde{N}_{w,G}(gN)\overline{\chi}(g)=\left\langle \tilde{N}_{w,G},\tilde{\chi}\right\rangle _{\nicefrac{G}{K}},
\]
where $\tilde{\chi}$ is the induced character on $\nicefrac{G}{K}$.
One easily sees that $\tilde{N}_{w,G}=|K|^{r-1}N_{w,\nicefrac{G}{K}}$,
and therefore 
\[
N_{w,G}^{\chi}=|K|^{r-1}N_{w,\nicefrac{G}{K}}^{\tilde{\chi}},
\]
and since $\nicefrac{G}{K}$ is abelian, $N_{w,G}^{\chi}\in\mathbb{Z}$.
\end{proof}

\begin{prop}
\label{prop:Cp_Cp-1} For every prime $p\geq5$ the group $G_{p}=C_{p}\rtimes C_{p-1}$
is semi-rational but does not satisfy the Amit-Vishne condition.\end{prop}
\begin{proof}
We start by proving that $G_{p}$ is semi-rational, and for that we
calculate the conjugacy classes: 

One can write $G_{p}=\left\langle s,t|\ t^{p}=s^{p-1}=1,\ s^{-1}ts=t^{k}\right\rangle $
for $k$ primitive in $\left(\nicefrac{\mathbb{Z}}{p\mathbb{Z}}\right)^{\times}$,
and get that 
\[
\left(s^{\alpha}t^{\beta}\right)^{-1}s^{x}\left(s^{\alpha}t^{\beta}\right)=s^{x}t^{\beta-k^{x}\beta},
\]
and since the function $m:\nicefrac{\mathbb{Z}}{(p-1)\mathbb{Z}}\rightarrow\nicefrac{\mathbb{Z}}{(p-1)\mathbb{Z}}$
defined by $m(\beta)=\beta(1-k^{x})$ is surjective, the set $\{s^{x},s^{x}t,\cdots,s^{x}t^{p-1}\}=s^{x}T$
is a conjugacy class, where $T=\left\langle t\right\rangle $. In
addition, $\left(s^{\alpha}t^{\beta}\right)^{-1}t^{x}\left(s^{\alpha}t^{\beta}\right)=t^{k^{\alpha}x}$,
and since $\alpha\mapsto k^{\alpha}x$ is also surjective, the conjugacy
classes are exactly 
\[
\{1\},\ T-\{1\},\ sT,\ \cdots,\ s^{p-2}T.
\]

Now, since $T\triangleleft G_{p}$ and $\nicefrac{G_{p}}{T}\overset{\sim}{=}C_{p-1}$,
there are $p-1$ one-dimensional characters of $G_{p}$ pulled from
$C_{p-1}$, and by column orthogonality one gets the full character
table of $G_{p}$:
\[
\begin{array}{ccccccc}
 & \{1\} & T-\{1\} & sT & s^{2}T & \cdots & s^{p-2}T\\
\chi_{1} & 1 & 1 & 1 & 1 & \cdots & 1\\
\vdots & \vdots & \vdots & \vdots & \vdots & \ddots & \vdots\\
\chi_{j} & 1 & 1 & e^{\frac{2\pi ij}{p-1}} & e^{\frac{4\pi ij}{p-1}} & \cdots & e^{\frac{2\pi ij(p-2)}{p-1}}\\
\vdots & \vdots & \vdots & \vdots & \vdots & \ddots & \vdots\\
\chi_{p-1} & 1 & 1 & e^{\frac{2\pi i(p-2)}{p-1}} & e^{\frac{4\pi i(p-2)}{p-1}} & \cdots & e^{\frac{2\pi i(p-2)^{2}}{p-1}}\\
\chi & p-1 & -1 & 0 & 0 & \cdots & 0
\end{array}
\]
By Proposition \ref{prop:character-condition-SR}, $G_{p}$ is semi-rational.

Next, we show that $G_{p}$ does not satisfy the Amit-Vishne condition:
Suppose that $\alpha$ is an automorphism of $G_{p}$ sending $s$
to $s^{-1}$. since $T$ is a normal $p$-Sylow subgroup, the image
of $t$ must be in $T$. Denote $\alpha(t)=t^{n}$, so 
\[
\alpha(ts)=\alpha(t)\alpha(s)=t^{n}s^{p-2}=s^{p-2}t^{k^{p-2}n}
\]
and on the other hand
\[
\alpha(ts)=\alpha(st^{k})=s^{p-2}t^{nk},
\]
which together gives $nk\equiv nk^{p-2}\left(\bmod p\right)$. Since
$n$ and $k$ are prime to $p$, we get $k^{p-3}\equiv1(\bmod p)$,
which is a contradiction, since $k$ is primitive, and $p\geq5$.
So there is no automorphism sending $s$ to $s^{-1}$, and $G_{p}$
does not satisfy the Amit-Vishne condition.
\end{proof}

For the next counter-example to Question \ref{ques:AV-condition}
we first prove the following condition:
\begin{prop}
\label{prop:comm-sqr-argument} If for $g,h\in G$ which generate
the same cyclic subgroup there exists a normal subgroup $N\triangleleft G$
such that $\nicefrac{G}{N}$ is semi-rational and such that $gN\subseteq O(g)$
and $hN\subseteq O(h)$ then $N_{w,G}(g)=N_{w,G}(h)$.\end{prop}
\begin{rem}
\label{rem:comm-square-argument}Here $O(g)$ denotes the orbit of
$g$ under $Aut(G)$.\end{rem}
\begin{proof}
Denote $H=\nicefrac{G}{N}$ and $\pi:G\rightarrow H$ the quotient
map. Then the following diagram commutes:
\[
\xymatrix{G^{r}\ar[r]^{\pi^{r}}\ar[d]_{w} & H^{r}\ar[d]^{w}\\
G\ar[r]_{\pi} & H
}
\]

By going through the upper branch, every element $x\in H$ has $|N|^{r}N_{w,H}(x)$
preimages in $G^{r}$, and by going through the lower branch $x$
has $\sum_{y\in\pi^{-1}(x)}N_{w,G}(y)$ preimages in $G^{r}$. Putting
this together for $x=gN$ and $x=hN$ we get 
\[
|N|^{r}N_{w,H}(gN)=\sum_{y\in gN}N_{w,G}(y)
\]
 and 

\[
|N|^{r}N_{w,H}(hN)=\sum_{y\in hN}N_{w,G}(y)
\]
respectively. But since $g,h$ generate the same subgroup in $G$,
so do $gN$ and $hN$ in $H$, giving $N_{w,H}(gN)=N_{w,H}(hN)$ since
$H$ is semi-rational. Therefore
\[
\sum_{y\in gN}N_{w,G}(y)=\sum_{y\in hN}N_{w,G}(y)\ \ \ \ (*)
\]
and since $gN\subseteq O(g)$, for every $y\in gN$ one has $N_{w,G}(g)=N_{w,G}(y)$,
so $(*)$ becomes
\[
|N|\cdot N_{w,G}(g)=|N|\cdot N_{w,G}(h),
\]
yielding $N_{w,G}(g)=N_{w,G}(h)$.
\end{proof}
This gives our second semi-rationality condition:
\begin{cor}
\label{cor:Ori-condition} Suppose there exists $N\triangleleft G$
such that $gN\subseteq O(g)$ for every $g\notin N$ , and that for
every $g,h\in N$ generating the same subgroup one has $N_{w,G}(g)=N_{w,G}(h)$.
Then $G$ is semi-rational.
\end{cor}

\begin{prop}
\label{prop:Cp^2Cp}For every prime $p\geq3$ the group $C_{p^{2}}\rtimes C_{p}=\left\langle s,t|\ t^{p^{2}}=s^{p}=1,\ s^{-1}ts=t^{p+1}\right\rangle $
is semi-rational but does not satisfy the Amit-Vishne condition.\end{prop}
\begin{proof}
The same kind of calculation done in Proposition \ref{prop:Cp_Cp-1}
shows that $gT=[g]\subseteq O(g)$ for every $g\notin T$, where $T=\left\langle t^{p}\right\rangle $.

For $g\in T$ we show that for every $y$ prime to the order of $g$
one can construct an automorphism of $G=C_{p^{2}}\rtimes C_{p}$ sending
$g$ to $g^{y}$, therefore giving $N_{w,G}(g)=N_{w,G}(g^{y})$. It
is enough to show this for $g=t^{p}$. Consider $\eta$ defined by
$t\mapsto t^{y}$ and $s\mapsto s$. Since $t^{y}s=st^{(p+1)y}=s\left(t^{y}\right)^{p+1}$,
$\eta$ extends uniquely to a homomorphism of $G$, and since $t^{y},s\in im(\eta)$
with $y$ prime to the order of $t$, $\eta$ is surjective and hence
an automorphism. Additionally, $\eta(t^{p})=\left(t^{y}\right)^{p}=\left(t^{p}\right)^{y}$,
as we wanted.

In conclusion, the conditions in Corollary \ref{cor:Ori-condition}
are satisfied, and therefore $G$ is semi-rational.

We show now that $G$ does not satisfy the Amit-Vishne condition,
as there is no automorphism sending $s$ to $s^{-1}$: Suppose $\alpha$
is such an automorphism. Since $\alpha(t)$ needs to be of order $p^{2}$,
one gets $\alpha(t)=t^{y}$ for some $1\leq y<p^{2}$ prime to $p$.
So
\[
\alpha(ts)=t^{y}s^{p-1}=s^{p-1}t^{(p+1)^{p-1}y}
\]
and 
\[
\alpha(ts)=\alpha(st^{p+1})=s^{p-1}t^{(p+1)y}
\]
 gives together $(p+1)^{p-1}y\equiv(p+1)y\ (\bmod p^{2})$, and since
$y,\,p+1$ are prime to $p$ we get that $(p+1)^{p-2}\equiv1\ (\bmod p^{2})$.
But 
\[
(p+1)^{p-2}\equiv1+p(p-2)\equiv1-2p\ (\bmod p^{2})
\]
which is a contradiction, since for $p\geq3$, $1-2p\not\equiv1\ (\bmod p^{2})$,
and therefore $G$ does not satisfy the Amit-Vishne condition.
\end{proof}

We use Corollary \ref{cor:Ori-condition} to provide another semi-rationality
result:
\begin{prop}
\label{prop:CpCq}The group $G=C_{p}\rtimes C_{q}$ is semi-rational
for $p,q$ primes and $p\equiv1(\bmod q)$.\end{prop}
\begin{proof}
We can write $G=\left\langle s,t|\ t^{p}=s^{q}=1,\ s^{-1}ts=t^{k}\right\rangle $
for some $k$ of order $q$ modulo $p$. By a straight forward computation,
the conjugacy class of $g$ is $gT=g\left\langle t\right\rangle $
for every $g\notin T$. 

In addition, we show that for every $g,h\in G$ generating the same
subgroup there is an automorphism sending $g$ to $h$. It is enough
to show that $t$ can be mapped to $t^{y}$ by an automorphism for
every $y\neq0$. This is true since the conditions $t\mapsto t^{y}$
and $s\mapsto s$ extend uniquely to an automorphism. This automorphism
surely sends $t$ to $t^{y}$. Therefore $N_{w,G}(g)=N_{w,G}(h)$
for every $g,h\in T$ generating the same subgroup.

The condition for Corollary \ref{cor:Ori-condition} hold, and therefore
$G$ is semi-rational.
\end{proof}

\subsection{The Finite Simple Group Case\label{sub:The-Finite-Simple}}

Even though the Amit-Vishne condition is not a necessary condition
for semi-rationality in the general case, for finite simple groups
Question \ref{ques:AV-condition} has a positive answer.

To show that we use a result of Lubotzky \cite{lubotzky2014images}:
\begin{thm}[\cite{lubotzky2014images}, Theorem 1]
\label{thm:Lubotzky}If A is a subset of a finite simple group $G$,
then there exists a word $w\in F_{k}$ for some $k$ such that $im(w)=A$
if and only if $1\in A$ and $\alpha(A)=A$ for every $\alpha\in Aut(G)$.\end{thm}
\begin{rem}
\label{rem:Lubotzky}One can even get $w\in F_{2}$, but unlike Theorem
\ref{thm:Lubotzky}, this requires the classification theorem of finite
simple groups.\end{rem}
\begin{prop}
\label{prop:FSG-AV-condition}For finite simple groups the Amit-Vishne
condition is sufficient and necessary. Namely, $G$ is semi-rational
if and only if for every $g,h\in G$ generating the same subgroup,
$\alpha(g)=h$ for some $\alpha\in Aut(G)$.\end{prop}
\begin{proof}
Suppose that $G$ does not satisfy the Amit-Vishne condition. Namely,
there are $g,g^{\prime}\in G$ generating the same subgroup such that
$\alpha(g)\neq g^{\prime}$ for every $\alpha\in Aut(G)$. Denote
by $O(g^{\prime})$ the orbit of $g^{\prime}$ under the action of
$Aut(G)$, and $A=G-O(g^{\prime})$. Then $1\in A$ since $1\notin O(g^{\prime})$,
and $\alpha(A)=A$ for every $\alpha\in Aut(G)$.

So by Theorem \ref{thm:Lubotzky} there exists $w\in F_{k}$ such
that $im(w)=A$. But $g\in A$ and $g^{\prime}\notin A$, and therefore
\[
N_{w,G}(g)>0=N_{w,G}(g^{\prime}).
\]
By Theorem \ref{thm:AV}, $N_{w,G}$ is not a generalized character,
so $G$ is not semi-rational.
\end{proof}

We use this result to prove that $PSL_{2}(p)$ is not semi-rational
for $p\geq11$:
\begin{prop}
\label{prop:PSL_2(p)}For every prime $p\geq11$ the group $PSL_{2}(p)$
is not semi-rational.\end{prop}
\begin{proof}
A direct computation shoes that $PSL_{2}(11)$ and $PSL_{2}(13)$
are not semi-rational, since for $w=xyx^{2}y^{3}$ the function $N_{w,PSL_{2}(11)}$
and $N_{2,PSL_{2}(13)}$ are not generalized characters of $PSL_{2}(11)$
and $PSL_{2}(13)$ respectively.

For $p\geq17$ we prove that $PSL_{2}(p)$ does not satisfy the Amit-Vishne
condition. Let $g$ be of the form $g=\left(\begin{array}{cc}
x\\
 & x^{-1}
\end{array}\right)\in PSL_{2}(p)$ for $x\in\mathbb{F}_{p}$ such that the order of $g$ is $\frac{p-1}{2}$.
This occurs for $x$ of order $\frac{p-1}{2}$ for $p\equiv3(\bmod4)$
and for $x$ of order $p-1$ for $p\equiv1(\bmod4)$. Choose $s$
prime to $\frac{p-1}{2}$ such that $g^{s}=\left(\begin{array}{cc}
x^{s}\\
 & x^{-s}
\end{array}\right)\notin\{g,g^{-1}\}$.

It is known that every automorphism of $PSL_{2}(p)$ is induced as
a conjugation by a matrix from $PGL_{2}(p)$ (\cite{wilson2009finite},
Section 3.3.4). Suppose that there exists an automorphism sending
$g$ to $g^{s}$. So we can write $g^{s}=A^{-1}gA$ for some $A\in PGL_{2}(p)$.
But looking at the action of $PGL_{2}(p)$ on $\mathbb{F}_{p}\cup\{\infty\}$,
the matrices $g,g^{s}$ fix only $0$ and $\infty$, and therefore
$A$ either fixes $0$ and $\infty$ or switches between the two.

If $A$ fixes $0$ and $\infty$, then $A$ is of the form $A=\left(\begin{array}{cc}
a\\
 & b
\end{array}\right)$, and if $A$ switches between $0$ and $\infty$, then $A$ is of
the form $A=\left(\begin{array}{cc}
 & a\\
b
\end{array}\right)$. But
\[
\left(\begin{array}{cc}
a\\
 & b
\end{array}\right)^{-1}\left(\begin{array}{cc}
x\\
 & x^{-1}
\end{array}\right)\left(\begin{array}{cc}
a\\
 & b
\end{array}\right)=\left(\begin{array}{cc}
x\\
 & x^{-1}
\end{array}\right)\neq g^{s},
\]
and 
\[
\left(\begin{array}{cc}
 & a\\
b
\end{array}\right)^{-1}\left(\begin{array}{cc}
x\\
 & x^{-1}
\end{array}\right)\left(\begin{array}{cc}
 & a\\
b
\end{array}\right)=\left(\begin{array}{cc}
x^{-1}\\
 & x
\end{array}\right)\neq g^{s}.
\]

Whence $PSL_{2}(p)$ does not satisfy the Amit-Vishne condition, and
being simple they are not semi-rational by Proposition \ref{prop:FSG-AV-condition}.\end{proof}
\begin{rem}
After the completion of this paper, it was pointed to us by R. Guralnick
that our Proposition \ref{prop:PSL_2(p)} overlaps with results in
\cite{guralnick2015rational}.
\end{rem}

\section{Further Questions\label{sec:Further-Questions}}
\begin{defn}
\label{def:abs-center}For a group $G$, denote $Z^{*}(G)=\{g\in G:\ \alpha(g)=g\ \forall\alpha\in Aut(G)\}$
and call it the absolute center of the group.
\end{defn}
Since the conditions for semi-rationality require powers of $g$ prime
to the order of $g$, only elements of order $3$ or more can provide
counter-example to semi-rationality. Therefore groups with absolute
center of exponent larger than $2$ are of natural interest.

Consider the groups $C_{p}\rtimes C_{q^{m}}$ for $p,q>2$ primes
with $p\equiv1(\bmod q)$ and $m\geq2$. One can write 
\[
C_{p}\rtimes C_{q^{m}}=\langle s,t|\ t^{p}=s^{q^{m}}=1,\ s^{-1}ts=t^{k}\rangle
\]
 where $k$ is of order $q$ modulo $p$. It can be shown by a straightforward
argument that $Z^{*}(C_{p}\rtimes C_{q^{m}})=\langle s^{q^{m-1}}\rangle$.
So the exponent of the absolute center of $C_{p}\rtimes C_{q^{m}}$
has exponent $q$.

Indeed several of the groups of the form $C_{p}\rtimes C_{q^{m}}$
are not semi-rational. Below is a list of such groups with a word
for which $N_{w,G}$ is not a generalized character, provided by direct
computations:

$C_{7}\rtimes C_{9}:\ w=x_{1}x_{2}x_{1}^{2}x_{2}^{5}$

$C_{11}\rtimes C_{25}:\ w=x_{1}x_{2}^{2}x_{1}^{4}x_{2}^{3}$

$C_{13}\rtimes C_{9}:\ w=x_{1}x_{2}^{2}x_{1}^{2}x_{2}^{7}$

$C_{19}\rtimes C_{9}:\ w=x_{1}x_{2}x_{1}^{2}x_{2}^{8}$

$C_{29}\rtimes C_{49}:\ w=x_{1}x_{2}x_{1}^{6}x_{2}^{6}$

$C_{31}\rtimes C_{9}:\ w=x_{1}^{2}x_{2}^{-7}x_{1}^{-8}x_{2}^{4}$

$C_{37}\rtimes C_{9}:\ w=x_{1}x_{2}^{-10}x_{1}^{-10}x_{2}^{4}.$

This raises the questions:
\begin{question}
\label{ques:CpCq^m}Is every group of the form $C_{p}\rtimes C_{q^{m}}$
not semi-rational for $m\geq2$?
\end{question}

\begin{question}
\label{ques:abs-center}Is every group with absolute center of exponent
larger than $2$ not semi-rational?
\end{question}

We note that $C_{19}\rtimes C_{9}$ is isomorphic to a normal subgroup
of $C_{19}\rtimes C_{18}$, and hence a normal subgroup of a semi-rational
group need not be semi-rational itself. A related question is:
\begin{question}
\label{ques:quotient-SR}Is every quotient of a semi-rational group
semi-rational itself?
\end{question}
\end{sloppy}

\bibliographystyle{alpha}
\bibliography{OnSemiRationalGroups}

\emph{Einstein Institute of Mathematics}

\emph{The Hebrew University of Jerusalem}

\emph{Jerusalem 91904}

\emph{Israel}

$ $

tzoor.plotnikov@mail.huji.ac.il
\end{document}